\documentclass{elsarticle}%
\usepackage{amsmath}
\usepackage{amsfonts}
\usepackage[colorlinks=true, urlcolor=blue]{hyperref}
\usepackage{amsthm}

\usepackage{amssymb}
\usepackage{bbold}
\usepackage{mathtools}
\usepackage{xcolor}
\usepackage{soul}
\usepackage{mathrsfs}
\usepackage{tabularx}
\usepackage[top=1in, bottom=1in, left=1in, right=1in]{geometry}
\usepackage{graphicx}
\providecommand{\U}[1]{\protect\rule{.1in}{.1in}}

\newtheorem{theorem}{Theorem}[section]
\newtheorem{lemma}{Lemma}[section]

\theoremstyle{definition}
\newtheorem{definition}{Definition}[section]
\theoremstyle{remark}
\newtheorem{remark}{Remark}[section]

\numberwithin{equation}{section}
\begin{document}
	\begin{frontmatter}

\title{On a Diophantine Equation Involving Lucas Numbers
}

		
		
		\author[1]{Seyran S. Ibrahimov}
		\ead{seyran.ibrahimov@emu.edu.tr}
		\author[1,2]{Nazim I. Mahmudov}
		\ead{nazim.mahmudov@emu.edu.tr}
		
		\address [1] {Department of Mathematics, Eastern Mediterranean University, Mersin 10, 99628, T.R. North Cyprus, Turkey}
		\address [2] {Research Center of Econophysics, Azerbaijan State University of Economics (UNEC), Istiqlaliyyat Str. 6, Baku 1001,Azerbaijan}
		
		

\begin{abstract}
Let $L_t$ denote the $t$-th Lucas number. We prove that the Diophantine equation
\[
L_{m}^{n+k} + L_{m}^{n} = L_r
\]
has no solutions in positive integers $r, m, n,$ and $k$ with $m \ge 2$. In particular, no Lucas number, except $L_0 = 2$, can be expressed as the sum of two distinct powers of another Lucas number. For $n = 1$, the proof relies on a precise factorization formula for the difference of two Lucas numbers together with the Carmichael Primitive Divisor Theorem. For $n \ge 2$, we apply Matveev’s theorem, combined with Legendre’s lemma and an exact divisibility property of powers of Lucas numbers.
\end{abstract}

		
\end{frontmatter}
	
\section{INTRODUCTION}
Let $a$ and $b$ be relatively prime integers, and let $(U_n)_{n\ge 0}$ and $(V_n)_{n\ge 0}$ denote the Lucas sequences of the first and second kinds, respectively. These sequences are defined recursively by
\[
U_0 = 0,\quad U_1 = 1,\quad U_n = aU_{n-1} + bU_{n-2} \quad \text{for } n \ge 2,
\]
and
\[
V_0 = 2,\quad V_1 = a,\quad V_n = aV_{n-1} + bV_{n-2} \quad \text{for } n \ge 2.
\]

To avoid trivial cases, we assume that $b \neq 0$ and that $\lambda/\kappa$ is not a root of unity, where $\lambda$ and $\kappa$ are the roots of the characteristic polynomial $x^2 - ax - b$. In particular, these assumptions imply that $\lambda \neq \kappa$, $\lambda \neq -\kappa$, the discriminant $D = a^2 + 4b \neq 0$, and that $U_n \neq 0$ and $V_n \neq 0$ for all $n \ge 1$.

It is well known that the Binet formulas
\[
U_n = \frac{\lambda^n - \kappa^n}{\lambda - \kappa}, \qquad
V_n = \lambda^n + \kappa^n, \quad \text{hold for all } n \ge 0.
\]
Many famous integer sequences arise as special cases of Lucas sequences. For example, the sequence of Fibonacci numbers is obtained from $(U_n)_{n \ge 0}$ by taking $a = b = 1$, and the sequence of Pell numbers is obtained from $(U_n)_{n \ge 0}$ by taking $a = 2$ and $b = 1$.

In particular, when $a = b = 1$, the sequence $(V_n)_{n \ge 0}$ coincides with the sequence of Lucas numbers $\{L_n\}_{n=0}^{\infty}$, which is defined by
\[
L_0 = 2, \quad L_1 = 1, \quad L_n = L_{n-1} + L_{n-2} \quad \text{for } n \ge 2.
\]

In this paper, we investigate the Diophantine equation
\begin{equation}\label{1.1}
L_{m}^{n+k} + L_{m}^{n} = L_r
\end{equation}
in positive integers $r, m, n,$ and $k$ with $m \ge 2$.

Moreover, the Binet formula for Lucas numbers is given by
\begin{equation}\label{Binet}
L_n = \alpha^{n} + \beta^{n}, \quad n \ge 0,
\end{equation}
where
\[
(\alpha, \beta) = \left( \frac{1 + \sqrt{5}}{2}, \frac{1 - \sqrt{5}}{2} \right).
\]

Using this representation, one obtains the inequalities (see \cite{4})
\begin{equation}\label{1.3}
\alpha^{n-1} \le L_n < \alpha^{n+1}, \quad n \ge 1.
\end{equation}

We now briefly review some related Diophantine problems from the literature. In \cite{2}, the authors solved a Diophantine equation involving Fibonacci numbers that is structurally similar to \eqref{1.1},
\[
F_n = F_l^{k}(F_l^{m} - 1),
\]
where \(n, m \geq 1\), \(k \geq 2\), and \(l \geq 3\).

Furthermore, in \cite{9}, Luca and Stănică proved that the equation
\begin{equation}\label{1.5}
w_n = p^a \pm p^b
\end{equation}
has only finitely many effectively computable positive integer solutions $(n, p, a, b)$, where $n \ge 3$, $a \ge \max\{2, b\}$, and $p$ is a prime. Here, $\{w_t\}_{t=0}^{\infty}$ denotes a Lucas sequence of the first $(U_n)_{n\ge 0}$ or second kind $(V_n)_{n\ge 0}$ whose characteristic polynomial has positive discriminant. When $L_m$ is prime, equation \eqref{1.1} becomes a special case of \eqref{1.5}. Nevertheless, even in this restricted situation, the present work provides a complete answer concerning the existence of solutions to equation \eqref{1.1}.

We also observe that, in the case $n = 1$, equation \eqref{1.1} reduces to
\begin{align*}
L_m^{k+1} + L_m = L_r.
\end{align*}

A more general form of this type of equation was investigated by Luca and Patel in \cite{7} in the context of Fibonacci numbers. More precisely, they studied the Diophantine equation
\begin{align*}
F_n \pm F_m = y^p,
\end{align*}
in integers $(n, m, y, p)$ with $p \ge 2$ and $n \equiv m \pmod{2}$. They proved that all such solutions satisfy either
$\max\{|n|, |m|\} \le 36,$ or $y = 0$ and $|n| = |m|$. The problem remains open in the case $n \not\equiv m \pmod{2}$.

For further related results, we refer the reader to \cite{3,5,8,13}.

\section{AUXILIARY RESULTS}
	
In this section, we provide a detailed overview of the fundamental tools employed in the proof of the main result.

The following result is well known and can also be established using formula \eqref{Binet}.
\begin{lemma}\label{factor}
If $r-m$ is even, then:
\[
L_r - L_m =
\begin{cases}
5 F_{\frac{r+m}{2}} F_{\frac{r-m}{2}}, & \text{if } r-m \equiv 0 \pmod{4}, \\[6pt]
L_{\frac{r+m}{2}} L_{\frac{r-m}{2}}, & \text{if } r-m \equiv 2 \pmod{4}.
\end{cases}
\]
\end{lemma} 
The following result can be found in \cite{4}.
\begin{lemma}\label{L-div}
\[
L_m \mid L_r
\quad \Longleftrightarrow \quad
r = m(2s-1)
\]
for some integer $s \ge 2$.
\end{lemma}
The following result is due to Carmichael \cite{1}.

\begin{lemma}[Carmichael Primitive Divisor Theorem]\label{Carmichael}
If $n \neq 1,6$, then $L_n$ has a primitive prime divisor; that is, 
there exists a prime $p \mid L_n$ such that
\[
p \nmid L_t \,\,\ \text{for all} \,\,\ t < n.
\]
\end{lemma}

We now state a lemma on exact divisibility by powers of Lucas numbers, which plays a crucial role in the proof of our main theorem. Results concerning divisibility by powers of recurrence sequences have attracted significant attention, particularly due to their application in Matijasevich's solution of Hilbert's 10th problem. In this direction, the most general results were established by Onphaeng and Pongsriiam \cite{11}. The lemma presented here enables us to considerably simplify the proof of our main result.

In this context, for integers $a \geq 2$, $k \geq 0$, and $b \geq 1$, we say that $a^k$ exactly divides $b$ by writing $a^k \parallel b$ if $a^k \mid b$ and $ a^{k+1} \nmid b$.

\begin{lemma}\label{lemma 2.4}\cite{12}
Let \( n, m, r \) be positive integers with \( n \geq 2 \) and \( m \geq 2 \).
 If \( L_m^n \parallel L_r \), then $
m \not\equiv 0 \pmod{3}$, \,\ $\frac{r}{m}$ is odd and \( L_m^{n-1} \parallel \frac{r}{m} \).
\end{lemma}

From equations \eqref{1.1}, we see that \( L_m^n \mid\mid L_r \). Applying Lemma \ref{lemma 2.4}, we find that for all \( n \geq 2 \) and \( m \geq 2 \), \( L_m^{n-1} \mid\mid \frac{r}{m} \). Then we get :
	
\begin{equation*}
m\cdot L_m^{n-1}\leq r
\end{equation*}
Hence, using inequalities \eqref{1.3}, we obtain
\begin{equation}\label{2.1}
\log(m) + (n - 1)(m - 1)\log \alpha \leq \log r,
\end{equation} 

	Next, we introduce some fundamental concepts from algebraic number theory.
	Let $z$ be an algebraic number of degree d with minimal polynomial
	
	\begin{align*}
		a_{0}x^{d}+a_{1}x^{d-1}+...+a_{d}=a_{0}\prod_{i=1}^{d}(x-z^{(i)})
	\end{align*}
	where the $a_{1}$, $a_{2}$,..., $a_{d}$ are relatively prime integers with $a_{0}> 0$ and $z^{(1)}$, $z^{(2)}$,...,  $z^{(d)}$  are conjugates of $z$.
	
\begin{definition}
The logarithmic height of $z$ is defined by
\begin{align*}
h(z)=\frac{1}{d}\bigg(\log a_{0}+\sum_{i=1}^{d} \log\big(\max\lbrace \vert z^{(i)}\vert,1\rbrace\big)\bigg)
\end{align*}
\end{definition}
	
	At this stage, we present the following lemma, which was introduced by Legendre in his book \cite{6}.
	
\begin{lemma}\label{lemma 2.5}
		
Let $x$ be a real number, with continued fraction expansion
\[
x = [a_0; a_1, a_2, a_3, \dots]
\] 
and let $p, q \in \mathbb{Z}$. If
\[
\left|x - \frac{p}{q} \right| < \frac{1}{2q^2}
\]
then $\frac{p}{q}$ is a convergent of the continued fraction of $x$. Furthermore, if $S$ and $N$ are non-negative integers such that $q_N > S$, then
\[
\left|x - \frac{p}{q} \right| > \frac{1}{(J(S) + 2)q^2},
\]
where $J(S) := \max \{a_i : i = 0, 1, 2, \dots, N\}$.
\end{lemma}
	
We will subsequently present a consequence of Matveev's theorem \cite{10}.
\begin{theorem}
Assume that $\beta_{1},\dots, \beta_{n}$ are positive algebraic numbers in a real algebraic number field $\mathbb{L}$ of degree $d$, $r_{1},\dots, r_{n}$ are rational integers, and
\begin{align*}
\Delta := \beta_{1}^{r_{1}}\dots \beta_{n}^{r_{n}}-1\not=0.
\end{align*}
then
\begin{align} \label{Matveev ineq}
\vert \Delta\vert > \exp \bigg(-1.4 \cdot 30^{n+3} \cdot n^{4.5} \cdot D^{2}(1+\log D)(1+\log T)A_{1}\dots A_{n}\bigg),
\end{align}
where $T\geq \max\lbrace \vert r_{1}\vert ,\dots, \vert r_{n}\vert\rbrace$, and $A_{j}\geq \max\lbrace Dh(\beta_{j}), \vert \log \beta_{j}\vert, 0.16\rbrace$, for all $j=1,\dots,n$.
	\end{theorem}
	
\section{ MAIN RESULTS}
\begin{theorem} There is no solution to equation~\eqref{1.1} in positive integers \( r, m, n, k \) with \( m \geq 2 \).
\end{theorem}
	
\begin{proof}	
For \( m = 1 \), the problem is trivial. Hence we may assume that \( m \ge 2 \). 
We first consider the case \( n = 1 \). Then
\begin{align}\label{n=1}
L_m^{k+1} = L_r - L_m.
\end{align}

It is clear that \( L_r \) is even, and therefore \( 3 \mid r \).
Since \( L_m \mid L_r \), by Lemma \ref{L-div} there exists an integer \( s \ge 2 \) such that
\[
r = m(2s - 1).
\]
Hence
\[
r - m = 2m(s - 1),
\]
so \( r - m \) is even. Therefore,
\[
r - m \equiv 0 \pmod{4}
\quad \text{or} \quad
r - m \equiv 2 \pmod{4}.
\]

Assume \( r - m \equiv 0 \pmod{4}. \)

By Lemma \ref{factor},
\[
L_r - L_m
=
5 F_{\frac{r+m}{2}} F_{\frac{r-m}{2}}.
\]
Thus,
\[
L_m^{k+1}
=
5 F_{\frac{r+m}{2}} F_{\frac{r-m}{2}}.
\]
Hence \( 5 \mid L_m^{k+1} \), and therefore \( 5 \mid L_m \).
However, it is well known that \( 5 \nmid L_n \) for any integer \( n \).
This is a contradiction.

Assume \( r - m \equiv 2 \pmod{4}. \)

Again by Lemma~\ref{factor},
\[
L_m^{k+1}
=
L_{\frac{r+m}{2}} L_{\frac{r-m}{2}}.
\]

Since \( r > m \), we have
\[
\frac{r+m}{2} > m.
\]
Assume first that
\[
\frac{r+m}{2} \ne 1, 6.
\]

Let \( p \) be a prime divisor of \( L_{\frac{r+m}{2}} \).
From \eqref{n=1}, it follows that \( p \mid L_m^{\,k+1} \), and hence \( p \mid L_m \).

However, by Lemma~\ref{Carmichael}, we know that
\[
p \nmid L_t \quad \text{for all } t < \frac{r+m}{2}.
\]
Since \( m < \frac{r+m}{2} \), it follows in particular that
\[
p \nmid L_m,
\]
which is a contradiction. Therefore, this case cannot occur.

\medskip
\noindent
It remains to consider the exceptional cases.

\medskip
If \( \frac{r+m}{2} = 1 \), then since \( r > m \ge 2 \), this is impossible.

Now suppose that \( \frac{r+m}{2} = 6 \). Since \( r = m(2s - 1) \), we obtain
\[
m(2s - 1) + m = 12,
\]
\[
2ms = 12,
\]
\[
ms = 6.
\]
Because \( m, s \ge 2 \), the possible pairs are
\[
(m,s) \in \{(2,3), (3,2)\}.
\]
This yields
\[
(m,r) \in \{(2,10), (3,9)\}.
\]

Since \( 3 \mid r \), the only admissible pair is \( (m,r) = (3,9) \).
However, for \( m = 3 \) and \( r = 9 \), equation~\eqref{n=1} becomes
\[
4^{k+1} = 76 - 4,
\]
which is impossible.

This completes the proof in the case \( n = 1 \).

From now on, we assume that $n \geq 2$. Under this assumption, we have
$L_r = L_m^{\,n+k} + L_m^{\,n} \geq L_2^{\,3} + L_2^{\,2} = 36.$
Consequently, we obtain \( r \geq 8 \). Using inequality \eqref{1.3}, we compare both sides of equation \eqref{1.1}, which yields 
\begin{align*}
\alpha^{(m+1)(n+k)+1} > L_m^n + L_m^{n+k} = L_r \geq \alpha^{r-1},
\end{align*}
\begin{align*}
\alpha^{(m-1)(n+k)} < L_m^n + L_m^{n+k} = L_r < \alpha^{r+1}.
\end{align*}
respectively imply the following inequalities:
\begin{align}\label{3.2}
r < 2 + (m+1)(n+k),
\end{align}
		
\begin{align}\label{3.3}
r > -1 + (m-1)(n+k).
\end{align}
At this stage, we rewrite equation \eqref{1.1} as 
\begin{align*}
\alpha^{r}-L_{m}^{n+k}=L_m^{n}-\beta^{r}
\end{align*}
implies
\begin{align*}
0<\alpha^r - L_m^{n+k}\leq L_m^n +\left| \beta\right|^r<L_m^n+0.022<1.0025L_m^n,
\end{align*}
Upon dividing both sides of the last inequality by \( L_{m}^{n+k} \), it follows that:
		
\begin{equation}\label{3.4}
0<\alpha^{r}L_m^{-(n+k)}-1< \frac{1.0025}{L_m^{k}}.
\end{equation}
We now apply  Matveev's theorem to obtain an upper bound for \( k \). Let us take:
\[
\beta_1 := \alpha, \quad \beta_2 := L_m,
\]
and the corresponding exponents:
\[
r_1 := r, \quad r_2 := -(n + k).
\]
Since \( \beta_1 \), and \( \beta_2 \) belong to the real quadratic number field \( \mathbb{L} = \mathbb{Q}(\sqrt{5}) \), we set \( D = 2 \). So, we take
\[
\Delta := \alpha^r L_m^{-(n+k)} - 1.
\]
Next, we verify that \( \Delta \neq 0 \). Suppose that \( \Delta = 0 \). Then we would have
\[
\alpha^r = L_m^{n+k},
\]
which implies that \( \alpha^r \in \mathbb{Q} \), leading to a contradiction. Since
\[
h(\beta_1) = \frac{1}{2} \log \alpha, \quad h(\beta_2) = \log L_m.
\]
We can choose:
\[		
A_1 := \log \alpha, \quad A_2 := 2 \log L_m.
\]
Considering inequality \eqref{3.3} and the fact that \( T \geq \max\{ r, n + k \} \), we can take \( T := 2r \).
By then combining inequalities \eqref{3.4} and \eqref{Matveev ineq}, we obtain
\begin{align*}
\frac{1.0025}{L_m^{k}}>\exp\bigg({-1.4\cdot30^5\cdot2^{7.5}(1+\log2)\log \alpha(1+\log(2r)\log L_m}\bigg)
\end{align*}
which implies that
\begin{align}\label{3.5}
k<5.38\cdot10^9(1+\log(2r)).
\end{align}
		
If we use inequality \eqref{3.5} together with the bounds \( m < 1 + \frac{\log r}{\log \alpha} \) and \( n \leq 1 + \frac{\log \frac{r}{2}}{\log \alpha} \), which are derived from \eqref{2.1}, in inequality \eqref{3.2}, we obtain
\begin{align}\label{3.6}
r<2+\left(2+\frac{\log r}{\log \alpha}\right)\left(1+\frac{\log \frac{r}{2}}{\log \alpha}+5.38\cdot10^9(1+\log(2r))\right)
\end{align}
		
which gives
\begin{align}\label{3.7}
r<1.1\cdot10^{13}.  
\end{align}
Then, by using this bound in inequality \eqref{3.5}, we get
\begin{align}\label{3.8}
k < 1.71 \cdot 10^{11}.
\end{align}
In the following step, we use the bound on \( r \) from \eqref{3.7} within inequality \eqref{2.1} to derive bounds for \( m \) and \( n \):
\begin{equation*}
\log(m) + (m - 1) \log(\alpha) \leq \log(1.1\cdot 10^{13})
\end{equation*}
\begin{equation*}
\log(2) + (n - 1) \log(\alpha) \leq \log(1.1\cdot 10^{13})
\end{equation*}
Consequently, it follows that:
\begin{equation}\label{3.9}
m\leq 55,\,n\leq 61.
\end{equation}
		
In this phase, we will reduce the upper bound of $k$.
Let 
\begin{align*}
\Gamma:=r\log\alpha-(n+k)\log L_m.
\end{align*}
Obviously, $\Delta = e^{\Gamma} - 1$.  
Since $\Delta > 0$, it follows that $\Gamma > 0$.  
Then, by using inequality \eqref{3.4} and the fact that $x < e^{x} - 1$ for $x \neq 0$, we obtain:
\begin{align}\label{3.10}
0<r\log\alpha-(n+k)\log L_m<\frac{1.0025}{L_m^{k}}
\end{align}
From inequality \eqref{3.10} we derive
\begin{align}\label{3.11}
0 < \left|\frac{\log L_m}{\log \alpha}-  \frac{r}{n+k}\right| < \frac{1.0025}{L_m^k \cdot (n+k) \log\alpha}
\end{align}
Assume that $k\geq 9$. By taking into account the conditions $m\geq2$ and $n\leq61$, we can write
\begin{align*}
L_m^{k}\geq3^{k}>\frac{2.005}{ \log\alpha}(61+k)\geq\frac{2.005}{ \log\alpha}(n+k)
\end{align*}
thus, so we have
\begin{align*}
0 < \left|\frac{\log L_m}{\log \alpha}-  \frac{r}{n+k}\right| < \frac{1.0025}{L_m^k \cdot (n+k) \log\alpha}<\frac{1}{2(n+k)^2}.
\end{align*}
Here, we apply Lemma \ref{lemma 2.5} by taking $x_m = \frac{\log L_m}{\log \alpha}$, $m=2,3,4,...,55$. Utilizing inequalities \eqref{3.8} and \eqref{3.9}
\bigskip
we obtain $n+k<61+1.71\cdot10^{11}$. Then, if we set \( S = 61+1.71\cdot10^{11} \), we need to find the integer \( N_m \) such that
\begin{align*}
q_{N_m}^{(m)} > 61+1.71\cdot10^{11},
\end{align*}
and take $J^{(m)}(S):=\max \{ a^{(m)}_i \mid i = 0, 1, 2, \dots, N_{m} \}, m=2,3,4,...,55$.
Then
\begin{align}\label{3.12}
\left|x_m - \frac{r}{n+k}\right| > \frac{1}{(J^{(m)}(S)+2)(n+k)^2},\,\ \text{for}\,\, m = 2, 3, \dots, 55.
\end{align}
Therefore, by combining inequalities \eqref{3.11} and \eqref{3.12}, we derive:
\begin{align}\label{3.13}
L_m^{k}<\frac{1.0025(J^{(m)}(S)+2)}{\log\alpha}(61+k) ,\,\ m = 2, 3, \dots, 55.
\end{align}
By taking \( m = 2, 3, 4, \ldots, 55 \) in \eqref{3.13}, and using the corresponding values of \( J^{(m)}(S) \), we find that \( k \leq 8 \) for all \( m \), which contradicts our assumption. Then, since \( k \leq 8 \), \( m \leq 55 \), and \( n \leq 61 \), it follows from inequality \eqref{3.2} that $r \leq 3865$. Substituting this bound for \( r \) into inequality \eqref{2.1}, we deduce:
\begin{align*}
m \leq 12, \quad n \leq 16.
\end{align*}
		
In conclusion, our problem is reduced to finding solutions under the conditions \( 2 \leq m \leq 12 \), \( 2 \leq n \leq 16 \), and \( k \leq 8 \), which imply that \( r\leq 313 \). A direct computational verification using Python confirms that equation \eqref{1.1} has no solutions in the specified range.
\end{proof}
	
We conclude this paragraph with the following remark.

\begin{remark}
We note that the upper bound for $r$ in \eqref{3.7} is rounded up; therefore, inequality \eqref{3.6} may not hold for values of $r$ close to this bound. Since $r$ is a variable in equation \eqref{1.1}, enlarging the bound does not restrict the set of possible solutions. Moreover, all subsequent inequalities have right-hand sides that are monotonically increasing in $r$, which ensures that this approach permits the use of more convenient bounds without excluding any potential solutions of equation \eqref{1.1}.
\end{remark}

\section{COMMENTS}
It would be interesting to study more general versions of the equation \eqref{1.1} for the Lucas sequences of the first and second kinds, namely,
\begin{align}\label{First}
U_m^{\,n+k} + U_m^{\,n} = U_r
\end{align}
and
\begin{align}\label{Second}
V_m^{\,n+k} + V_m^{\,n} = V_r.
\end{align}

The main questions are:
\begin{enumerate}
    \item Are the solutions finite for any special cases of the Lucas sequences of the first and second kinds?  
    \item Are the solutions of \eqref{First} and \eqref{Second} finite (that is, $a$ and $b$ not fixed)?
\end{enumerate}

So far, the Fibonacci, Lucas cases of these equations have been solved. We believe that, for fixed $a$ and $b$, all solutions can be obtained using the same methods we employed. Furthermore, exact divisibility by powers of integers in Lucas sequences, as discussed in \cite{11}, may play an important role in addressing equations \eqref{First} and \eqref{Second}.

We leave these questions open for the reader.

\section*{Author Contributions}
All authors have accepted responsibility for the entire content of this manuscript and approved its submission.

\section*{Data Availability}

Some of the calculations in this study were performed using Python. The codes necessary to reproduce these calculations are publicly available at Zenodo: \url{https://doi.org/10.5281/zenodo.17479021}.

\section*{Conflict of interest}

The authors state no conflict of interest.

\end{document}